\documentclass[10pt]{amsart}

\usepackage{comment}
\usepackage{mathtools}
\usepackage{amsmath}
  \usepackage{paralist}
  \usepackage{graphics} %% add this and next lines if pictures should be in esp format
  \usepackage{epsfig} %For pictures: screened artwork should be set up with an 85 or 100 line screen
\usepackage{graphicx}  \usepackage{epstopdf}%This is to transfer .eps figure to .pdf figure; please compile your paper using PDFLeTex or PDFTeXify.

 \usepackage[colorlinks=true]{hyperref}
\hypersetup{urlcolor=blue, citecolor=red}
\newtheorem{theorem}{Theorem}[section]

\newtheorem*{main*}{Main Theorem}
\newtheorem{lemma}[theorem]{Lemma}
\newtheorem{proposition}[theorem]{Proposition}

\newtheorem{question}[theorem]{Question}
\theoremstyle{definition}
\newtheorem{definition}[theorem]{Definition}
\newtheorem{remark}[theorem]{Remark}

\newtheorem*{TheoremA}{Theorem A}
\newtheorem*{TheoremB}{Theorem B}
\newtheorem*{TheoremC}{Theorem C}
\newtheorem*{TheoremD}{Theorem D}

\newtheorem*{CorollaryA1}{Corollary A.1}
\newtheorem*{CorollaryA2}{Corollary A.2}

\newtheorem*{CorollaryC1}{Corollary C.1}
\newtheorem*{CorollaryC2}{Corollary C.2}
\newtheorem*{CorollaryC3}{Corollary C.3}
\newtheorem*{CorollaryD1}{Corollary D.1}
\newtheorem*{CorollaryD2}{Corollary D.2}

\title[Unstable entropies and Dimension Theory of Partially Hyperbolic Systems]{Unstable entropies and Dimension Theory of Partially Hyperbolic Systems}%[Unstable Entropies of subsets]
\def\a{\alpha}
\def\b{\beta}
   
\def\d{\delta}   %\def\C{\Gamma}
   
 \def\e{\epsilon}

\def\ae{\text{-a.e.}\ }

\def\P{{\mathcal P}}

\def\CW{{\mathcal W}}

\def\CU{{\mathcal U}}
\def\CE{{\mathcal E}}

\def\M{{\mathcal M}}

\def\CG{{\mathcal G}}

\def\RR{{\mathbb R}}
\def\NN{{\mathbb N}}

\def\diam{\mathop{\hbox{{\rm diam}}}}

\def\loc{{\mathop{\hbox{\footnotesize  \rm loc}}}}

\def\disp{\displaystyle}

\author[]{Xueting Tian and Weisheng Wu$^{*}$}


\address{School of Mathematical Sciences, Fudan University, Shanghai, 200433, P.R. China}
 \email{xuetingtian@fudan.edu.cn}

\address{Department of Applied Mathematics, College of Science, China Agricultural University, Beijing, 100083, P.R. China}
 \email{wuweisheng@cau.edu.cn}

\thanks{$^*$Corresponding author.}

\begin{document}

\maketitle
\markboth{Unstable entropies and dimension theory}
{X. Tian and W. Wu}
\renewcommand{\sectionmark}[1]{}

\begin{abstract}
In this paper we define unstable topological entropy for any subsets (not necessarily compact or invariant) in partially hyperbolic systems as a Carath\'{e}odory dimension characteristic, motivated by the work of Bowen and Pesin etc.  We then establish some basic results in dimension theory for Bowen unstable topological entropy, including an entropy distribution principle and a variational principle in general setting. As applications of this new concept, we study unstable topological entropy of saturated sets and extend some results in \cite{Bo, PS2007}. Our results give new insights to the multifractal analysis for partially hyperbolic systems.
\end{abstract}

\section{Introduction}

%%%%%%%%%%%%%%%%%%%%%%%%%%%%%%%%%%%%%%%%%%%%%%%
%%%%%%%%%%%%%%%%%%%%%%%%%%%%%%%%%%%%%%%%%%%%%%%
%\section{Definitions and statements of main results}\label{Sdef}
%\setcounter{equation}{0}
%%%%%%%%%%%%%%%%%%%%%%%%%%%%%%%%%%%%%%%%%%%%%%%
%%%%%%%%%%%%%%%%%%%%%%%%%%%%%%%%%%%%%%%%%%%%%%%

The study of dimension theory (in particular, multifractal analysis) of hyperbolic dynamical systems
has drawn the attention of many researchers (\cite{Pesin,BaSa,BaSch,Takens,TV} and many others).
The general concept of multifractal analysis is to decompose the phase space into subsets of points which have a
similar dynamical behavior and to describe the size of these subsets from
the geometrical or topological viewpoint.
Sets with similar dynamical behavior include the basin set of an invariant measure or general saturated sets \cite{Bo,PS2007}, recurrent and dense sets \cite{T16,CT},  non-dense sets \cite{Urban,DT,Wu1,Wu2}, level sets and irregular sets of Birkhoff ergodic average \cite{Pesin,PP,BaSau,BaSa,BaSch,BSS,BC,TV,PS2007,LLST,Thompson2012,KS,Takens}, level sets and irregular sets of Lyapunov exponents \cite{BG,PeSa,FengH,Tian2}, which have been studied a lot
by using various measurements such as Hausdorff dimension, topological entropy or pressure, Lebesgue measure and distributional chaos etc. Here the topological entropy used was introduced by  Bowen \cite{Bo} to characterize the dynamical complexity of arbitrary sets which are not necessarily compact nor invariant from the perspective of ``dimensional'' nature.
Pesin and Pitskel$'$ further developed Bowen's approach and proposed the notion of topological pressure of a general subset (cf. \cite{PP}). In Pesin's book \cite{Pesin}, a general Carathe\'{o}dory construction is presented, and the notion of Carathe\'{o}dory dimension is defined. Among Carathe\'{o}dory dimensions are Hausdorff dimension, $q$-dimension, topological pressure including topological entropy, etc.

The dimension theory of hyperbolic dynamical systems has been  well-studied with mature methods including thermodynamic formalism approach (\cite{BaSa,BD} etc.) and orbit gluing approach (\cite{TV,PS2007} etc.), for which the uniqueness of equilibrium states and certain weak form of specification are essential, respectively. However, the dimension theory of dynamical systems beyond uniform hyperbolicity is far away from being well-studied. There are only few developments for nonuniformly hyperbolic or expanding systems (without zero Lyapunov exponents) \cite{LLST,TianVar}.  Another important type of dynamical systems beyond uniform hyperbolicity is partially hyperbolic diffeomorphisms which admit zero Lyapunov exponents. However, to the best of our knowledge, there are no results on dimension theory of partially hyperbolic diffeomorphisms.

The crucial techniques of thermodynamic formalism and orbit gluing approach in hyperbolic case may not work in general partially hyperbolic systems. Either the uniqueness of equilibrium states or weak forms of specification are very rare. For instance, a partially hyperbolic (but not hyperbolic) system can not have specification robustly since  a system with specification robustly if and only if it is transitive Anosov \cite{SSY}. And a partial hyperbolic system with two saddles of different indexes can not have specification \cite{SVY}.  Nevertheless,  there are some progress in positive direction. For example, some special partially hyperbolic systems including Ma\~{n}\'{e}'s examples have uniqueness of equilibrium states  established in \cite{CliTho,CFT}. This is mainly because such systems satisfy non-uniform versions of specification and expansiveness. As for orbit gluing approach, we notice that a quasi-shadowing property and a center specification property are obtained in \cite{HZZ} and \cite{WZ} which might be helpful.

Partial hyperbolicity includes the hyperbolic part and the center part. Since the latter may have zero exponents, one can firstly consider dynamical complexity caused by the hyperbolic part.
In their seminal papers \cite{LY1,LY2}, Ledrappier and Young
gave a characterization of the SRB measures,
and more generally,
 established the so-called dimension formula for all invariant measures.
They
 introduced a hierarchy of metric entropies $h_i=h_i(f)$,
each of which corresponds to different largest positive Lyapunov exponents,
and is regarded as the entropy caused by different levels of
unstable manifolds. A natural problem is then to study the \emph{topological entropy} caused by different hierarchy of
unstable manifolds and its relationship with the metric entropy considered by Ledrappier and Young. For $C^1$-partially hyperbolic diffeomorphisms, Hu, Hua and Wu recently introduced in \cite{HHW} a concept called unstable topological entropy to characterize dynamical complexity of the whole system caused by unstable manifolds and established a variational principle relating the unstable topological entropy with the unstable metric entropy. Another useful result is the upper semi-continuity of unstable metric entropy (cf. \cite{HHW,Yang}).

Inspired by the work of Bowen and Pesin \cite{Bo,PP, Pesin}, in the present paper we propose a concept called \emph{Bowen unstable topological entropy} to study dynamical complexity of general subsets in partially hyperbolic systems. In particular, the Bowen unstable topological entropy of the whole space coincides with the unstable topological entropy of the system in \cite{HHW}. We then establish some basic results in dimension theory for Bowen unstable topological entropy, including an entropy distribution principle, and a variational principle for any compact (not necessarily invariant) subset between its Bowen unstable topological entropy and unstable metric entropy of probability measures supported on this set. As applications of this new concept, we study Bowen unstable topological entropy of saturated sets and extend some results in \cite{Bo, PS2007} to the context of unstable entropies. The main ingredient in our multifractal analysis is the adapted orbit gluing approach for unstable entropies. The thermodynamic formalism approach to multifractal analysis based on Bowen unstable pressure will appear in a forthcoming paper \cite{TW}.

\subsection{Results}
 Through  this paper we consider partially hyperbolic diffeomorphisms.   Let $M$ be an $n$-dimensional smooth, connected, compact Riemannian manifold without boundary and let $f: M \to M$ be a $C^1$-diffeomorphism. $f$ is called a \emph{partially hyperbolic diffeomorphism} (abbreviated as \emph{PHD}) if there exists a nontrivial $Df$-invariant splitting of the tangent bundle $TM= E^s \oplus E^c \oplus E^u$ into so-called stable, center, and unstable distributions, such that all unit vectors $v^{\sigma} \in E^\sigma(x)$ ($\sigma=s, c, u$) with $x\in M$ satisfy
\begin{equation*}
\|D_xfv^s\| < \|D_xfv^c\| < \|D_xfv^u\|,
\end{equation*}
and
\begin{equation*}
\|D_xf|_{E^s(x)}\| <1, \ \ \ \text{\ and\ \ \ \ } \|D_xf^{-1}|_{E^u(x)}\| <1,
\end{equation*}
for some suitable Riemannian metric on $M$. The distributions $E^s$, $E^c$, $E^u$ are H\"{o}lder continuous over $M$. The stable distribution $E^s$ and unstable distribution $E^u$ are integrable: there exist so-called stable and unstable foliations $W^s$ and $W^u$ respectively such that $TW^s=E^s$ and $TW^u=E^u$. See for example \cite{RRU} for the basic properties of partially hyperbolic diffeomorphisms. %$f$ is called \emph{Anosov} if it is partially hyperbolic with trivial center.

The first aim of this paper is to set up a theory for unstable entropies of general subsets over PHDs in the framework of Carathe\'{o}dory dimensions. We will define the notions of \emph{Bowen unstable topological entropy} and \emph{uppper capacity unstable topological entropy} for arbitrary subsets which are not necessarily compact or $f$-invariant. Usually it is hard to compute Bowen unstable topological entropy directly. An effective way is to estimate it through an entropy distribution principle. We define unstable metric entropy for any Borel probability measure on $M$ which is not necessarily $f$-invariant. See Section 2 for the detailed definitions and notations.  Denote by $\M(M)$ ($\M_f(M)$) the set of Borel probability measures ($f$-invariant Borel probability measures respectively) on $M$. Our first result is the following Billingsley type theorem or entropy distribution principle:
\begin{TheoremA}\label{ThmA}
Let $f: M \to M$ be a $C^1$-PHD, $Z$ a Borel subset of $M$, $\mu \in \M(M)$, and $0<s<\infty$.
\begin{enumerate}
  \item If $\underline{h}_\mu^u(f,x)\leq s$ for every $x\in Z$, then $h_{\text{\text{B}}}^u(f,Z)\leq s$.
  \item If $\underline{h}_\mu^u(f,x)\geq s$ for every $x\in Z$ and $\mu(Z)>0$, then $h_{\text{\text{B}}}^u(f,Z)\geq s$.
  \end{enumerate}
\end{TheoremA}

As a corollary, we can generalize Theorem 1 in \cite{Bo} for Bowen unstable entropy:
\begin{CorollaryA1}\label{AA1}
Let $f: M \to M$ be a $C^1$-PHD. Then $h_{\text{\text{B}}}^u(f,Z)\geq \underline{h}_\mu^u(f)$ for any $\mu \in \M(M)$ with $\mu(Z)=1$.
\end{CorollaryA1}

The unstable topological entropy of $f$ defined in \cite{HHW} is the uppper capacity unstable topological entropy of $M$, which turns out to coincide with Bowen unstable topological entropy of $M$, i.e., $h_{\text{\text{UC}}}^u(f,M)=h_{\text{\text{B}}}^u(f,M)$. More generally, we have:

\begin{CorollaryA2}\label{AA2}
Let $f: M \to M$ be a $C^1$-PHD and $Z\subset M$ compact and $f$-invariant. Then $h_{\text{\text{B}}}^u(f,Z)=h_{\text{\text{UC}}}^u(f,Z)$.
\end{CorollaryA2}

A variational principle relating unstable topological entropy of $f$ and unstable metric entropy of invariant measures is
given in \cite{HHW}. In \cite{FH}, a variational principle is established relating Bowen topological entropy of any compact subset and metric entropy of Borel probability measures supported on it (see Theorem \ref{subsetvp} below). We find our new concept of Bowen unstable topological entropy of compact subsets also satisfies a similar variational principle so that it is suitable to study dynamical complexity of subsets in partial hyperbolicity.

\begin{TheoremB}\label{ThmB}
Let $f: M \to M$ be a $C^1$-PHD and $K\subset M$ a nonempty compact subset. Then
$$h_{\text{\text{B}}}^u(f,K)=\sup\{\underline{h}_{\mu}^u(f): \mu\in \M(M) \  \text{and\ } \mu(K)=1\}.$$
\end{TheoremB}

Our second aim is to study Bowen unstable topological entropy of saturated sets. Denote $\CE_n(x):=\frac{1}{n}\sum_{i=0}^{n-1}\d_{f^i(x)}$ where $\d_y$ is the Dirac measure at $y\in M$.
The limit-point set of $\{\CE_n(x)\}_n$ is a compact connected subset $V_f(x)\subset \M_f(M)$. For $K\subseteq \mathcal M_f(M)$ which is nonempty, compact and connected, the set $G_K=\{x:V_f(x)=K\}$ is called the \emph{saturated set of $K$}. The existence of saturated sets was firstly found by Sigmund in \cite{Sigmund} for hyperbolic systems (see \cite{LST} for nonuniformly hyperbolic case). In order to compare the complexity of different saturated sets, Pfister and Sullivan used topological entropy in \cite{PS2007} to estimate the complexity of a saturated set $G_K$  and established some equalities by using the metric entropy of measures in $K.$ Similar as in \cite{PS2007} we can first consider $$\prescript{K}{}{G}=\{x:V_f(x)\cap K\neq \emptyset\}$$ for general PHDs.
\begin{TheoremC}\label{ThmC}
Let $f: M \to M$ be a $C^1$-PHD, $K\subset \M_f(M)$ a closed subset.
Then
  $$h^u_{\text{B}}(f, \prescript{K}{}{G})\leq \sup\{h_\mu^u(f): \mu\in K\}.$$
\end{TheoremC}

\begin{CorollaryC1}\label{CA1}
Let $f: M \to M$ be a $C^1$-PHD and $G_\mu:=\{x\in M: V_f(x)=\mu\}.$
Then
$$h^u_{\text{B}}(f, G_\mu)\leq h_\mu^u(f) \text{\ \ for any\ }\mu \in \M_f(M).$$
Furthermore, if $\mu$ is ergodic, then $h^u_{\text{B}}(f, G_\mu)= h_\mu^u(f)$.
\end{CorollaryC1}
In particular, Corollary C.1 extends Bowen's result on topological entropy of the set of generic points with respect to an ergodic measure in \cite{Bo} to the setting of unstable entropy.

We obtain the following inequality for saturated set $G_K$.
\begin{CorollaryC2}\label{CA2}
Let $f: M \to M$ be a $C^1$-PHD and $K\subset \M_f(M)$ be non-empty, connected and compact. Then
$$h^u_{\text{B}}(f, G_K)\leq \inf\{h_\mu^u(f): \mu\in K\}.$$
\end{CorollaryC2}

Birkhoff Ergodic Theorem is a classical and basic way to study dynamical orbits by describing asymptotic behavior from the probabilistic viewpoint of a given observable function. Some concepts in the theory of multifractal analysis such as
level sets of Birkhoff ergodic average, derive from the Birkhoff Ergodic Theorem.  It was established a variational principle in \cite{BaSa} for level sets of hyperbolic systems (also see \cite{TV,PS2007} for results obtained by using specification-like properties). Given a continuous function $\varphi: M\to \RR$, for any $a\in  \mathbb{R}$,  consider the level set
\[R_{\varphi}(a) :=\left\{ {x\in M}: \lim_{n\to\infty}\frac1n\sum_{i=0}^{n-1}\varphi(f^ix)=a\right\}.\]

 \begin{CorollaryC3}\label{CorC3}
 Let $f: M \to M$ be a $C^1$-PHD and $\varphi:M\rightarrow \mathbb{R}$ a continuous function. Then
 for any $a\in \mathbb{R}$,
 $$h^u_{\text{B}}(f,R_\varphi(a))\leq\sup\left\{h^u_\mu(f): \int \varphi d\mu =a, \ \mu\in \M_f(M)\right\}.$$
 \end{CorollaryC3}

When the paper was being written we found that recently Ponce \cite{Pon} used Bowen's original ideas \cite{Bo}
to define unstable topological entropy of subsets and get Corollaries A.1, A.2, and the ``furthermore'' part of Corollary C.1 for $C^{1+\a}$-PHDs. Here we use the methods of Carathe\'{o}dory dimension to define unstable topological entropy of subsets and develop more general results.

\subsection{Questions}
From above results, our Bowen unstable topological entropy seems to be a
suitable concept to study lots of things in partial hyperbolicity. And one can define refined Bowen unstable topological entropy  when the unstable direction can be decomposed into several sub-directions. One natural question similar as  the motivation of  \cite{LY1,LY2}  on dimension formula is  whether the (whole) unstable topological entropy equals to the sum of Bowen unstable topological entropy on sub-unstable bundles.  However, we do not know the techniques of the present paper are enough or not.

It was proved in  \cite{PS2007} that if the system has  $g$-almost product  property and uniform separation property (see  \cite{PS2007} for their detailed definitions), then  for any nonempty, weak$^*$-compact and connected set $K \subset {\mathcal M}_f(M)$,
$$h_{\text{B}}(f, G_K)\geq \inf\{h_\mu(f): \mu\in K\},$$
where $G_K:=\{x\in M: V_f(x)=K\}.$ These assumptions hold for topological mixing locally maximal hyperbolic  sets. For the unstable entropy considered in the present paper, a natural question arises:
\begin{question}\label{QueD}
Let $f: M \to M$ be a $C^1$-PHD. Then  for any non-empty, connected and compact subset
$K\subset \M_f(M)$, do we have
$$h^u_{\text{B}}(f, G_K)\geq \inf\{h_\mu^u(f): \mu\in K\}?$$
\end{question}

\begin{remark}
This question  is important because the positive answer will imply that the inequality in Corollary C.3 is an equality, by following the proof of Proposition 7.1 in \cite{PS2007}. The positive answer also implies that the irregular set has full unstable topological entropy by following the way in \cite{Tian}; see Corollary D.1 for a partial result on the completely irregular set.
\end{remark}

We give a partial answer for this question. Firstly, an unstable version of uniform separation is naturally true for PHDs. The detailed definition of unstable uniform separation property is given at the beginning of Section 5. Secondly, we give an estimate of topological entropy on $G_K$ by unstable metric entropy from below. As we mentioned, partial hyperbolicity has the quasi-shadowing property \cite{HZZ} and the center specification \cite{WZ}, but the shadowing orbit may be not a true orbit so that one can not use them to get nonempty saturated sets. Thus here we still assume $g$-almost product  property.
%\end{definition}
\begin{TheoremD}\label{ThmD}
Let $f: M \to M$ be a $C^1$-PHD.
\begin{enumerate}
  \item Then $f$ has unstable uniform separation property.
  \item If we assume further that $f$ has  $g$-almost product property, then for any non-empty, connected and compact subset
$K\subset \M_f(M)$, $$h_{\text{B}}(f, G_K)\geq \inf\{h_\mu^u(f): \mu\in K\}.$$
\end{enumerate}
\end{TheoremD}

Theorem D applies to the study of the irregular set in multifractal analysis. We can also derive from Theorem D(1) a formula for the unstable metric entropy of an ergodic measure, which extends the formula given by Pfister-Sullivan (\cite{PS2007}). These are Corollaries D.1 and D.2 stated in the last section.

This paper is organized as follows. In Section 2 we define unstable topological entropies using Carathe\'{o}dory construction, and then obtain some basic properties of them. In Section 3, we study the dimension theory of unstable entropies, proving Theorem A and B and their corollaries. In Sections 4 and 5, we study the unstable topological entropies of saturated sets, and prove Theorems C and D respectively. Corollaries D.1 and D.2 are proved at the end of the paper.

\section{Unstable entropies}
\subsection{Unstable topological entropies of subsets}
We define the notions of Bowen unstable topological entropy and uppper capacity unstable topological entropy of arbitrary subsets in PHD. Our definitions and notations follow the Carathe\'{o}dory construction presented in \cite{Pesin}.

Suppose that $\CU$ is a finite open cover of $M$.
Denote $\diam (\CU):=\max\{\diam (U): U\in \CU\}$.
For $n\geq 1$, we denote by $\CW_n(\CU)$ the collections of strings
$\textbf{U}=U_{i_0}\cdots U_{i_{n-1}}$ with $U_i\in \CU$.
Given $\textbf{U}\in \CW_n(\CU)$, denote by $m(\textbf{U})=n$ the length of $\textbf{U}$, and define
\[X(\textbf{U}):=\{x\in M: f^j(x)\in U_{i_{j}}, j=0, 1, \cdots, m(\textbf{U})-1\}.\]
Let $Z\subset M$ be a nonempty set. For $s\in \RR$, define
\[\M_N^s(\CU, Z):=\inf_{\CG}\sum_{\textbf{U}\in \CG}\exp(-sm(\textbf{U})),\]
where the infimum is taken over all $\CG\subset \bigcup_{j\geq N}\CW_j(\CU)$ that covers $Z$,
i.e., $\bigcup_{\textbf{U}\in \CG}X(\textbf{U})\supset Z$.
Clearly $\M_N^s(\CU, \cdot)$ is a finite outer measure on $M$, and it determines a dimension-like characteristic as follows.
Since $\M_N^s(\CU, Z)$ increases as $N$ increases, we can define $\M^s(\CU, Z)=\lim_{N\to \infty}\M_N^s(\CU, Z)$ and
\[h_{\text{\text{B}}}(f,\CU, Z):=\inf\{s: \M^s(\CU, Z)=0\}=\sup\{s: \M^s(\CU, Z)=+\infty\}.\]
Set $h_{\text{\text{B}}}(f, Z)=\sup_{\CU}h_{\text{\text{B}}}(f,\CU, Z)$
where $\CU$ runs over all finite open covers of $M$.

Similarly, for $s\in \RR$, define
\[R_N^s(\CU, Z):=\inf_{\CG}\sum_{\textbf{U}\in \CG}\exp(-sN),\]
where the infimum is taken over all $\CG\subset \bigcup_{j\geq 1}\CW_j(\CU)$ that covers $Z$ and $m(\textbf{U})=N$.
We can define
\begin{equation*}\label{e:upper}
R^s(\CU, Z)=\limsup_{N\to \infty}R_N^s(\CU, Z)
\end{equation*}
 and
\[h_{\text{\text{UC}}}(f,\CU, Z):=\inf\{s: R^s(\CU, Z)=0\}=\sup\{s: R^s(\CU, Z)=+\infty\}.\]
Set $h_{\text{\text{UC}}}(f, Z)=\sup_{\CU}h_{\text{\text{UC}}}(f,\CU, Z)$
where $\CU$ runs over all finite open covers of $M$.

We denote by $d^u$ the metric induced by the Riemannian structure
on the unstable manifolds and let $W^u(x,\delta)$ be the open ball inside $W^u(x)$ centered at $x$
of radius $\delta$ with respect to the metric $d^u$. Let $d$ denote the metric induced by the Riemannian structure on $M$. If $\d$ is small enough, then $d^u$ is equivalent to $d$ restricted to $\overline{W^u(x,\delta)}$.
\begin{definition}\label{Defutopent1}
The \emph{Bowen unstable topological entropy} of $Z$ with respect to $f$ is defined by
\begin{equation*}
%\begin{aligned}
h^u_{\text{\text{B}}}(f,Z)
=\lim_{\delta \to 0}\sup_{x\in M}h_{\text{\text{B}}}(f, \overline{W^u(x,\delta)}\cap Z),
%\end{aligned}
\end{equation*}

The \emph{upper capacity unstable topological entropy} of $Z$ with respect to $f$ is defined by
\begin{equation*}
%\begin{aligned}
h^u_{\text{\text{UC}}}(f,Z)
=\lim_{\delta \to 0}\sup_{x\in M}h_{\text{\text{UC}}}(f, \overline{W^u(x,\delta)}\cap Z).
%\end{aligned}
\end{equation*}
\end{definition}

\begin{remark}
Replacing the $\limsup$ by $\liminf$ in \eqref{e:upper}, one can define the \emph{lower capacity unstable topological entropy} of $Z$ with respect to $f$ similarly. Readers interested in lower capacity can follow Pesin's book \cite{Pesin} to get corresponding results.
\end{remark}
The Bowen unstable topological entropy can be defined in an alternative way.
If $y\in W^u(x)$, let
$d^u_{n}(x,y)=\max _{0 \leq j \leq n-1}d^u(f^j(x),f^j(y))$,
and $B^u_{n}(x,\epsilon)=\{y\in W^u(x): d^u_{n}(y,x) \leq\epsilon\}$
be the \emph{$(n,\epsilon)$ $u$-Bowen ball} around $x$.
For $Z\subset \overline{W^u(x,\delta)}$, $s\geq 0$, $N\in \NN $ and $\e>0$, define
\[\M_{N,\e}^s(Z):=\inf \sum_{i}\exp(-sn_i),\]
where the infimum is taken over all finite or countable families $\{B^u_{n_i}(x_i,\e)\}$ such that $x_i\in \overline{W^u(x,\delta)}$,
$n_i\geq N$ and $\bigcup_{i}B^u_{n_i}(x_i,\e)\supset Z$.
As  {$\M_{N,\e}^s(Z)$} increases as $N$ increases, we can define $\M_\e^s(Z)=\lim_{N\to \infty}\M_{N,\e}^s(Z)$ and $\M^s(Z)=\lim_{\e\to 0}\M_\e^s(Z)$. Then define
\[\bar{h}_{\text{B}}(f, Z):=\inf\{s: \M^s(Z)=0\}=\sup\{s: \M^s(Z)=+\infty\}.\]

\begin{proposition}\label{redefine}
$\bar{h}_{\text{B}}(f, Y)=h_B(f,Y)$ for any $Y\subset \overline{W^u(x,\delta)}$. Hence for any $Z\subset M$,
$$h_{\text{\text{B}}}(f, Z)=\lim_{\delta \to 0}\sup_{x\in M}\bar{h}_{\text{\text{B}}}(f, \overline{W^u(x,\delta)}\cap Z).$$
\end{proposition}
\begin{proof}
We note that $d^u$ is equivalent to $d$ on $\overline{W^u(x,\delta)}$. Then the first statement is Proposition 5.2 in \cite{Cli} and a detailed proof is given there. The second statement follows immediately.
\end{proof}

As for upper capacity unstable topological entropy, there are two alternative ways to define it.
For $\epsilon>0$ and $n\in \NN$, a subset $S\subset \overline{W^u(x,\delta)}\cap Z $ is called an
\emph{$(n,\epsilon)$ $u$-separated set} of $\overline{W^u(x,\delta)}\cap Z$ if
$x,y\in S, x\neq y$, then $d^u_n(x,y)>\e$.
A set $E \subset W^u(x)$ is called an \emph{$(n,\epsilon)$ $u$-spanning set}
of $\overline{W^u(x, \delta)}\cap Z$ if
$\overline{W^u(x, \delta)} \cap Z\subset \bigcup_{y\in E}B^u_{n}(y,\epsilon)$.
Let $N^u(Z,\epsilon,n,x,\delta)$ be the maximal
cardinality of an $(n,\epsilon)$ $u$-separated set of $\overline{W^u(x,\delta)}\cap Z$ and
$S^u(Z,\epsilon,n,x,\delta)$ the minimal cardinality of an
$(n,\epsilon)$ $u$-spanning set of $\overline{W^u(x, \delta)}\cap Z$.
It is standard to verify that
$$
N^u(Z,2\epsilon,n,x,\delta) \leq S^u(Z,\epsilon,n,x,\delta)
\leq N^u(Z,\epsilon,n,x,\delta).
$$
Then checking similarly as in Proposition \ref{redefine}, we have
\begin{proposition}
\begin{equation*}
\begin{aligned}
 h_{\text{\text{UC}}}(f, \overline{W^u(x,\delta)}\cap Z)
&=\lim_{\epsilon \to 0}\limsup_{n\to \infty}\frac{1}{n}\log N^u(Z,\epsilon,n,x,\delta)\\
&=\lim_{\epsilon \to 0}\limsup_{n\to \infty}\frac{1}{n}
\log S^u(Z,\epsilon,n,x,\delta).
\end{aligned}
\end{equation*}
\end{proposition}

The following lemma tells that in Definition \ref{Defutopent1} we do not have to
let $\d\to 0$. The proof is an adaption of the one of Lemma 4.1 \cite{HHW} to
Bowen and upper capacity unstable topological entropies.
\begin{lemma}\label{smalldelta}
For $Z\subset M$,
\[h^u_{\text{B}}(f,Z)=\sup_{x\in M}h_{\text{B}}(f, \overline{W^u(x,\delta)}\cap Z)\]
and
$$h^u_{\text{UC}}(f,Z)=\sup_{x\in M}h_{\text{UC}}(f, \overline{W^u(x,\delta)}\cap Z)$$
for any $\delta >0$.
\end{lemma}

At last, we collect some basic properties of unstable topological entropies. One can follow Theorems 11.2 and 11.3 in Pesin's book \cite{Pesin} for a proof.
\begin{proposition}\label{basic}
\begin{enumerate}
  \item   $h^u_{\text{B}}(f,Z_1)\subset h^u_{\text{B}}(f,Z_2)$ and $h^u_{\text{UC}}(f,Z_1)\subset h^u_{\text{UC}}(f,Z_2)$ if $Z_1\subset Z_2$.
  \item (Countable stability) $h^u_{\text{B}}(f,\bigcup_{i\in \NN} Z_i)=\sup_{i\in\NN}h^u_{\text{B}}(f,Z_i).$
   \item $h^u_{\text{UC}}(f,\bigcup_{i\in \NN} Z_i)\geq \sup_{i\in\NN}h^u_{\text{UC}}(f,Z_i), h^u_{\text{UC}}(f,\bigcup_{i=1}^n Z_i)= \max_{i=1}^nh^u_{\text{UC}}(f,Z_i).$
  \item  $h^u_{\text{B}}(f,Z)=h^u_{\text{B}}(f,fZ)$ and $h^u_{\text{UC}}(f,Z)=h^u_{\text{UC}}(f,fZ)$.
  \end{enumerate}
\end{proposition}

\subsection{Unstable metric entropy of measures}
In \cite{HHW}, the authors give a new definition for the unstable metric entropy $h_{\mu}^u(f)$ which coincides with the one considered by Ledrappier and Young \cite{LY2}, by using measurable partitions subordinate to unstable manifolds that can be obtained by refining a finite Borel partition into pieces of unstable leaves. We recall this definition as follows.

For a partition $\a$ of $M$, let $\a(x)$ denote the element of $\a$
containing $x$.
If $\a$ and $\b$ are two partitions such that $\a(x)\subset \b(x)$
for all $x\in M$, we then write $\a \geq \b$ or $\b\leq \a$.
The partition $\a\vee\b$ is defined by $(\a\vee\b)(x)=\a(x)\cap\b(x)$ for any $x\in M$.
For a measurable partition $\b$, we denote
$\disp\b_m^n=\vee_{i=m}^n f^{-i}\b$.  In particular,
$\disp\b_0^{n-1}=\vee_{i=0}^{n-1} f^{-i}\b$.

Take $\e_0>0$ small.
Let $\P=\P_{\e_0}$ denote the set of finite Borel partitions of $M$
whose elements have diameters smaller than or equal to $\e_0$, that is,
$\diam \a:=\sup\{\diam A: A\in \a\}\le \e_0$.
For each $\b\in \P$ we can define a finer partition $\eta$ such that
$\eta(x)=\b(x)\cap W^u_\loc(x)$ for each $x\in M$, where $W^u_\loc(x)$
denotes the local unstable manifold at $x$ whose size is
greater than the diameter $\e_0$ of $\beta$.
Clearly $\eta$ is a measurable partition satisfying $\eta\ge \b$.
%We denote such partition $\b$ by $\eta^\sharp$, that is, for any
%$\eta^\sharp\in \P$ and $\eta(x)=\eta^\sharp(x)\cap W^u_\loc(x)$ for any
%$x\in M$.
Let $\P^u=\P^u_{\e_0}$ denote the set of partitions $\eta$ obtained this way.

A partition $\xi$ of $M$ is said to be
\emph{subordinate to unstable manifolds} of $f$ with respect
to a measure $\mu$ if for $\mu$-almost every $x$,
$\xi(x)\subset W^u(x)$
and contains an open neighborhood of $x$ in $W^u(x)$.
It is clear that if $\a\in \P$ such that
$\mu(\partial \a)=0$ where $\partial \a:=\cup_{A\in \a} \partial A$,
then the corresponding $\eta$ given by $\eta(x)=\a(x)\cap W^u_\loc(x)$
is a partition subordinate to unstable manifolds of $f$.

Given a measure $\mu$ and measurable partitions $\a$ and $\eta$, let
$$H_\mu(\a|\eta):=-\int_M \log \mu_x^\eta(\alpha(x))d\mu(x)$$
denote the conditional entropy of $\a$ given $\eta$ with respect to $\mu$,
where $\{\mu_x^\eta: x\in M\}$ is a family of conditional measures of $\mu$ relative to $\eta$.
See \cite{R} for the discussion of conditional measures.

\begin{definition}\label{Defuentropy}
The \emph{conditional entropy of $f$ with respect to a measurable partition $\a$
given $\eta\in \P^u$} is defined as
$$h_\mu(f, \alpha|\eta)
=\limsup_{n\to \infty}\frac{1}{n}H_\mu(\alpha_0^{n-1}|\eta).
$$
The \emph{conditional entropy of $f$ given $\eta\in \P^u$}
is defined as
$$h_\mu(f|\eta)
=\sup_{\alpha \in \P}h_\mu(f, \alpha|\eta).
$$
and the \emph{unstable metric entropy of $f$} is defined as
\[
h_\mu^u(f)=\sup_{\eta\in \P^u}h_\mu(f|\eta).
\]
\end{definition}

By Corollary A.2 and Theorem B in \cite{HHW}, we have
\begin{proposition}
If $\mu$ is ergodic, then for any $\a\in \P$, $\eta\in \P^u$ subordinate to unstable manifolds and any $\e>0$,
\begin{equation*}
h_\mu^u(f)=h_\mu(f|\eta)=h_\mu(f, \alpha|\eta)%=H_\mu(f^{-1}\xi|\xi)
=\lim_{n\to \infty}-\frac{1}{n}\log\mu_x^{\eta}(B^u_{n}(x,\epsilon)) \quad
\mu\ae x.
\end{equation*}
\end{proposition}

It is natural to give a definition for unstable metric entropy of any Borel probability measure on $M$ as follows.
\begin{definition}
Let $\mu$ be a Borel probability measure on $M$, $\eta\in \P^u$ subordinate to unstable manifolds and $\e>0$ small enough. Define
the \emph{unstable metric entropy} of $\mu$ as
\begin{equation*}
\underline{h}_\mu^u(f)=\int_M\underline{ h}_\mu^u(f,x) d\mu, \quad \overline{h}_\mu^u(f)=\int_M\overline{ h}_\mu^u(f,x) d\mu,
\end{equation*}
where
\begin{equation*}
\begin{aligned}
\underline{h}_\mu^u(f,x)&=\lim_{\e\to 0}\liminf_{n\to \infty}-\frac{1}{n}\log\mu_x^{\eta}(B^u_{n}(x,\epsilon)), \\
\overline{h}_\mu^u(f,x)&=\lim_{\e\to 0}\limsup_{n\to \infty}-\frac{1}{n}\log\mu_x^{\eta}(B^u_{n}(x,\epsilon)).
\end{aligned}
\end{equation*}
\end{definition}
If $\mu$ is $f$-invariant and ergodic, it is proved in (9.2) and (9.3) in \cite{LY2} that
\begin{equation*}
\underline{h}_\mu^u(f, x)=\overline{h}_\mu^u(f, x)=h^u_\mu(f)
\end{equation*}
for $\mu \ae x$ and thus $h^u_\mu(f)=\underline{h}_\mu^u(f)=\overline{h}_\mu^u(f)$.

If $\mu$ is $f$-invariant, we argue $h^u_\mu(f)=\underline{h}_\mu^u(f)=\overline{h}_\mu^u(f)$ as follows. Let $\mu=\int_{\mathcal{M}^e_f(M)}\nu d\tau(\nu)$ be the unique ergodic decomposition
where $\tau$ is a probability measure supported on $\mathcal{M}^e_f(M)$, the subspace of $\M_f(M)$ consisting of ergodic measures.
Since $\mu \mapsto h_\mu^u(f)$ is affine and upper semi-continuous by Propositions 2.14 and 2.15 in \cite{HHW}, then
\begin{equation}\label{e:ergodic}
h_\mu^u(f)=\int_{\mathcal{M}^e_f(M)}h_\nu^u(f) d\tau(\nu)
\end{equation}
by a classical result in convex analysis (cf. Fact A.2.10 on p. 356 in \cite{Do}). Let $\mathfrak{g}$ be the sub-$\sigma$-algebra consisting of all invariant subsets and let $\zeta$ be a measurable partition such that $\mathfrak{g}$ is equivalent mod zero to the $\sigma$-algebra generated by $\zeta$. Let $\{\mu_x\}$ be a family of conditional probability measures associated with $\zeta$. Then there is an invariant set $N_1\subset M$ with $\mu(N_1)=1$ such that for every $x\in N_1$, $\mu_x$ is $f$-invariant and ergodic. Moreover, \eqref{e:ergodic} implies $h_\mu^u(f)=\int_{M}h_{\mu_x}^u(f) d\mu(x)$. Recall the crucial fact that $\eta$ refines $\zeta$ by arguing similarly to the proof of Proposition 2.6 in \cite{LS}. Then there exists a measurable set $N_2\subset N_1$ with $\mu(N_2)=1$ such that for every $x\in N_2$, $\{\mu_y^\eta\}$ is a family of conditional probability measures associated with $\eta$ in the space $(N_2\cap \zeta(x), \mu_x)$. Therefore $\underline{h}^u_{\mu}(f,x)=\underline{h}^u_{\mu_x}(f,x)=h_{\mu_x}^u(f)$ and $\overline{h}^u_{\mu}(f,x)=\overline{h}^u_{\mu_x}(f,x)=h_{\mu_x}^u(f)$ for $\mu \ae x\in N_2$. Integrating it we have $\underline{h}_\mu^u(f)=\overline{h}_\mu^u(f)=\int_{M}h_{\mu_x}^u(f) d\mu(x)=h^u_\mu(f)$.

\begin{lemma}\label{independence}
Both $\underline{h}_\mu^u(f,x)$ and $\overline{h}_\mu^u(f,x)$ are independent of the choice of $\eta\in \P^u$ for $\mu$-a.e. $x$, hence so are $\underline{h}_\mu^u(f)$ and $\overline{h}_\mu^u(f)$.
\end{lemma}
\begin{proof}
We only prove the result for $\underline{h}_\mu^u(f,x)$, and the result for $\overline{h}_\mu^u(f,x)$ can be proved similarly.
Note that if $\eta_1, \eta_2 \in \P^u$, then $\eta_1\vee \eta_2\in \P^u$. So without of generality, we assume that $\eta_1\leq \eta_2$. For $\mu$-a.e. $x$, if $n$ is large enough, then $B^u_{n}(x,\epsilon)\subset \eta_2(x)\subset \eta_1(x)$. Then $\mu_x^{\eta_1}(B^u_{n}(x,\epsilon))=\mu_x^{\eta_1}(\eta_2(x))\cdot \mu_x^{\eta_2}(B^u_{n}(x,\epsilon))$. It is easy to see that
\begin{equation*}
\begin{aligned}
&\lim_{\e\to 0}\liminf_{n\to \infty}-\frac{1}{n}\log\mu_x^{\eta_1}(B^u_{n}(x,\epsilon)) \\
=&\lim_{\e\to 0}\liminf_{n\to \infty}-\frac{1}{n}\log\Big(\mu_x^{\eta_1}(\eta_2(x))\cdot \mu_x^{\eta_2}(B^u_{n}(x,\epsilon))\Big)\\
=&\lim_{\e\to 0}\liminf_{n\to \infty}-\frac{1}{n}\log\mu_x^{\eta_2}(B^u_{n}(x,\epsilon)).
\end{aligned}
\end{equation*}
So $\underline{h}_\mu^u(f,x)$ is independent of the choice of $\eta\in \P^u$.
\end{proof}
It is not hard to extend the variational principle for unstable entropies in \cite{HHW} to get
\begin{theorem}\label{vp}
Let $f: M \to M$ be a $C^1$-PHD and $Z\subset M$ an $f$-invariant and compact subset. Then
\begin{equation*}
\begin{aligned}
h_{\text{\text{UC}}}^u(f,Z)&=\sup\{h_\mu^u(f): \mu \in \M_f(M) \text{\ with\ }\mu(Z)=1\}\\
&=\sup\{h_\mu^u(f): \mu \in \M_f^e(M) \text{\ with\ }\mu(Z)=1\}.
\end{aligned}
\end{equation*}
\end{theorem}

Adapting the method in the proof of the variational principle in \cite{HHW}, we have
\begin{lemma}\label{vplemma}
Let $S_n\subset \overline{W^u(x,\d)}$ be a sequence of $(n,\e)$ $u$-separated subsets and define
\[\nu_n:=\frac{1}{n\#S_n}\sum_{z\in S_n}\sum_{i=0}^{n-1}\delta_{f^i(z)}.\]
Assuming $\lim_{n\to \infty}\nu_n=\mu$, then
\[\limsup_{n\to \infty}\frac{1}{n}\log \#S_n\leq h^u_\mu(f),\]
where we denote by $\#S_n$ the cardinality of the set $S_n$.
\end{lemma}

\section{Dimension theory for unstable entropies}

\subsection{Proof of Theorem A}
We prove Theorem A, the entropy distribution principle for Bowen unstable entropy. It is well known that the entropy distribution principle is a basic result in dimension theory, which allows us to estimate topological entropy of subsets through entropies of measures.
\begin{proof}[Proof of Theorem A(1)]
By Lemma \ref{smalldelta},
\[h^u_{\text{B}}(f,Z)=\sup_{x\in M}h_{\text{B}}(f, \overline{W^u(x,\delta)}\cap Z)\]
for any small $\d>0$. Pick $\d>0$ small enough.
It is enough to prove that $h_{\text{B}}(f, \overline{W^u(x,\delta)}\cap Z)\leq s$ for any $x\in M$.
Fix $x\in  M$. Take $\eta\in \P^u$ subordinate to unstable manifolds such that $\overline{W^u(x,\delta)}\subset \eta(x)$.
Let $\rho>0$ be sufficiently small. For each $k\geq 1$, define
\begin{equation}\label{e:levelset}
\begin{aligned}
Z_k:=\left\{y\in Z: \liminf_{n\to \infty} \frac{-\log \mu_y^\eta(B_n^u(y, \e))}{n}\leq s+\rho, \text{\ for all\ }\e\in \left(0,\frac{1}{k}\right)\right\}.
\end{aligned}
\end{equation}
Then $Z=\bigcup_{k\geq 1}Z_k$. Now fix $k\geq 1$ and $0<\e<\frac{1}{3k}$.
For each $y\in Z_k$, there exists an increasing sequence $n_j(y)\to \infty$ such that
$$\mu_y^\eta(B_{n_j(y)}^u(y, \e))\geq e^{-n_j(y)(s+\rho)} \quad \text{\ for all\ }j\geq 1.$$
For any $N\geq 1$, $\mathcal{F}=\{B_{n_j(y)}^u(y, \e): y\in Z_k\cap \overline{W^u(x,\delta)}, n_j(y)\geq N\}$
is a cover of $Z_k\cap \overline{W^u(x,\delta)}$. The following lemma is very similar to the Vitali Covering lemma.
Its proof is a slight modification of that of Lemma 1 in \cite{MW} and is omitted here.
\begin{lemma}\label{vitali}
Let $\e>0$, $X\subset \overline{W^u(x,\delta)}$, and $\mathcal{\text{B}}(\e)=\{B_n^u(y,\e): y\in X, n=1,2, \cdots\}$. For any family $\mathcal{F}\subset \mathcal{\text{B}}(\e)$,
there exists a (not necessarily countable) subfamily $\mathcal{G}\subset \mathcal{F}$ consisting of disjoint balls such that
$$\bigcup_{B\in \mathcal{F}}B\subset \bigcup_{B^u_n(y,\e)\in \mathcal{G}}B_n^u(y,3\e).$$
\end{lemma}
By Lemma \ref{vitali}, there exists
a subfamily $\mathcal{G}=\{B_{n_i}^u(y_i, \e)\}_{i\in I}$ of $\mathcal{F}$ consisting of disjoint balls such that
$$Z_k\cap \overline{W^u(x,\delta)}\subset \bigcup_{i\in I}B_{n_i}^u(y_i, 3\e).$$
Note that $y_i \in \overline{W^u(x,\delta)}\subset \eta(x)$. Then $\mu_{y_i}^\eta=\mu_x^\eta$.
We have
$$\mu_x^\eta(B_{n_i}^u(y_i, \e))\geq e^{-n_i(s+\rho)}\quad \text{\ for all\ }i\in I.$$
The index set $I$ is at most countable since $\mu_x^\eta$ is a probability Borel measure and each $B_{n_i}^u(y_i, \e)$ has positive $\mu_x^\eta$-measure. It follows that
\begin{equation*}
\begin{aligned}
\M^{s+\rho}_{N,3\e}(\overline{W^u(x, \delta)}\cap Z_k)\leq \sum_{i\in I}e^{-n_i(s+\rho)}\leq \sum_{i\in I}\mu_x^\eta(B_{n_i}^u(y_i, \e))\leq 1.
\end{aligned}
\end{equation*}
Letting $N\to \infty$ and then $\e\to 0$, we have
$$\M^{s+\rho}(\overline{W^u(x, \delta)}\cap Z_k)\leq 1.$$
Thus
$h_{\text{B}}(f, \overline{W^u(x,\delta)}\cap Z_k)\leq s+\rho$. Then
$$h_{\text{B}}(f, \overline{W^u(x,\delta)}\cap Z)=\sup_{k\geq 1}h_{\text{B}}(f, \overline{W^u(x,\delta)}\cap Z_k)\leq s+\rho.$$
Since $\rho$ is arbitrary,
$$h_{\text{B}}(f, \overline{W^u(x,\delta)}\cap Z)\leq s$$
for all $x\in M$ as desired.
 \end{proof}

\begin{remark}
Now let us assume $\overline{h}_\mu^u(f,x)\leq s$ for every $x\in Z$ instead in Theorem A(1).
Then we have \eqref{e:levelset} with ``$\liminf$'' replaced by ``$\limsup$''.
A similar argument proves that
$$\lim_{k\to \infty}h^u_{\text{UC}}(f, Z_k)\leq s.$$
However, we can not conclude that $h^u_{\text{UC}}(f, Z)\leq s$ as upper capacity unstable entropy is  not necessarily countably stable (see Proposition \ref{basic} (3)).
\end{remark}

\begin{proof}[Proof of Theorem A(2)]
Fix $\rho>0$. For each $k\geq 1$, define
\begin{equation*}
\begin{aligned}
Z_k:=\left\{x\in Z: \liminf_{n\to \infty} \frac{-\log \mu_x^\eta(B_n^u(x, \e))}{n}\geq s-\rho, \text{\ for all\ }\e\in \left(0,\frac{1}{k}\right)\right\}.
\end{aligned}
\end{equation*}
Then $Z_k\subset Z_{k+1}$, and $Z=\bigcup_{k=1}^\infty Z_k$ since $\underline{h}_\mu^u(f,x)\geq s$ for every $x\in Z$. We fix some $k\geq 1$ such that $\mu(Z_k)>\frac{1}{2}\mu(Z)>0$.

For each $N\geq 1$, let
\begin{equation*}
\begin{aligned}
Z_{k,N}:=\left\{x\in Z_k: \frac{-\log \mu_x^\eta(B_n^u(x, \e))}{n}\geq s-\rho, \text{\ for all\ } n\geq N \text{\ and\ } \e\in \left(0,\frac{1}{k}\right)\right\}.
\end{aligned}
\end{equation*}
Then $Z_{k,N} \subset Z_{k,N+1}$ and $Z_k=\bigcup_{N=1}^\infty Z_{k,N}$. Then there exists $N_0\in \NN$ such that $\mu(Z_{k,N_0})>\frac{1}{2}\mu(Z_k)>0$. One has
$$\mu_x^\eta(B_n^u(x, \e))\leq e^{-n(s-\rho)} \quad \text{\ for all\ }x\in Z_{k,N_0}, n\geq N_0 \text{\ and\ } 0<\e<\frac{1}{k}.$$
Hence, there exists $x\in M$ such that
$\mu_x^\eta(Z_{k,N_0})=\mu_x^\eta(Z_{k,N_0} \cap \eta(x))>0$.
Fix such $x$.
Note that if $y \in \eta(x)$, then $\mu_y^\eta=\mu_x^\eta$.
We have
$$\mu_x^\eta(B_n^u(y, \e))\leq e^{-n(s-\rho)} \quad \text{\ for all\ }y\in Z_{k,N_0}\cap \eta(x), n\geq N_0 \text{\ and\ } 0<\e<\frac{1}{k}.$$

Fix a sufficiently large $N>N_0$. For any cover $\mathcal{G}=\{B^u_{n_i}(x_i, \e/2): n_i\geq N, x_i\in W^u(x)\}$ of $Z_{k,N_0} \cap\eta(x)$,
we assume without
loss of generality that
$B^u_{n_i}(x_i, \e/2) \cap Z_{k,N_0} \cap\eta(x) \neq \emptyset$ for any $i$.
Let $y_i$ be an arbitrary point in $B^u_{n_i}(x_i, \e/2) \cap Z_{k,N_0} \cap\eta(x)$.
By triangle inequality, $B^u_{n_i}(x_i, \e/2)\subset B^u_{n_i}(y_i, \e)$. We have
\begin{equation*}
\sum_i e^{-n_i(s-\rho)} \geq \sum_i\mu_x^\eta(B^u_{n_i}(y_i, \e))\geq \sum_i\mu_x^\eta(B^u_{n_i}(x_i, \e/2))\geq \mu_x^\eta(Z_{k,N_0})>0.
\end{equation*}
Now we take $\delta>0$ such that $W^u(x, \delta)\supset \eta(x)$.
It follows that
\begin{equation*}
\begin{aligned}
&\M^{s-\rho}(\overline{W^u(x, \delta)}\cap Z)\geq \M^{s-\rho}_{N,\e/2}(\overline{W^u(x, \delta)}\cap Z)\\
\geq &\M^{s-\rho}_{N,\e/2}(\overline{W^u(x, \delta)}\cap Z_{k,N_0})\geq \M^{s-\rho}_{N,\e/2}(\eta(x)\cap Z_{k,N_0})\geq \mu_x^\eta(Z_{k,N_0})>0.
\end{aligned}
\end{equation*}
Thus $h_{\text{B}}(f, \overline{W^u(x, \delta)}\cap Z)\geq s-\rho$. Letting $\rho\to 0$,
$h_{\text{B}}(f, \overline{W^u(x, \delta)}\cap Z)\geq s.$
By Lemma \ref{smalldelta},
\[h^u_{\text{B}}(f,Z)=\sup_{x\in M}h_{\text{B}}(f, \overline{W^u(x,\delta)}\cap Z).\]
We have $h^u_{\text{B}}(f,Z)\geq h_{\text{B}}(f, \overline{W^u(x, \delta)}\cap Z)\geq s$.
\end{proof}

\begin{proof}[Proof of Corollary A.1]
Let $Z\subset M$ and $\mu \in \M(M)$ with $\mu(Z)=1$.
Define
$$Z_\delta:=\{x\in Z: \underline{h}_\mu^u(f,x)\geq \underline{h}_\mu^u(f)-\d\}.$$
As $\underline{h}_\mu^u(f)=\int_M\underline{h}_\mu^u(f,x) d\mu$, it follows that $\mu(Z_\d)>0$ for any $\d>0$.
By Theorem A(2), $h_B^u(f,Z_\d)\geq \underline{h}_\mu^u(f)-\d$. Since $Z_\d\subset Z$, $h_B^u(f,Z)\geq \underline{h}_\mu^u(f)-\d$.
Letting $\d\to 0$, we get $h_B^u(f,Z)\geq \underline{h}_\mu^u(f)$ as desired.
\end{proof}

\begin{remark}
In fact, Theorem A(2) is equivalent to Corollary A.1. Indeed, the other direction can be proved as follows.
Assume that $\underline{h}_\mu^u(f,x)\geq s$ for every $x\in Z$ and $\mu(Z)>0$. Let
$\nu=\mu|_{Z}$ be the conditional measure of $\mu$ on $Z$. Then $\nu$ is a probability measure on $Z$ such that
for $\nu \ae x\in Z$,
\begin{equation*}
\begin{aligned}
\underline{h}_\nu^u(f,x)&=\lim_{\e\to 0}\liminf_{n\to \infty}-\frac{1}{n}\log\nu_x^{\eta}(B^u_{n}(x,\epsilon))\\
&=\lim_{\e\to 0}\liminf_{n\to \infty}-\frac{1}{n}\Big( \log\mu_x^{\eta}(B^u_{n}(x,\epsilon))-\log \mu(Z) \Big) \\
&=\underline{h}_\mu^u(f,x)\geq s.
\end{aligned}
\end{equation*}
It follows that $\underline{h}_\nu^u(f)=\int_M\underline{h}_\nu^u(f,x) d\nu\geq s$.
Using Corollary A.1, we know $h_B^u(f,Z)\geq \underline{h}_\nu^u(f)\geq s$
as desired.
\end{remark}

\begin{proof}[Proof of Corollary A.2]
By definition, $h_{\text{\text{B}}}^u(f,Z)\leq h_{\text{\text{UC}}}^u(f,Z)$ for any $Z\subset M$.
We only need to prove the other direction.
Now let $Z\subset M$ be compact and $f$-invariant. By Theorem \ref{vp},
$$h_{\text{\text{UC}}}^u(f,Z)=\sup\{h_\mu^u(f): \mu \in \M_f(M) \text{\ with\ }\mu(Z)=1\}.$$
On the other hand, by Theorem A(1), for any $\mu\in \M_f(M)$ with $\mu(Z)=1$, $h_\mu^u(f)\leq h_{\text{\text{B}}}^u(f,Z)$.
It follows that $h_{\text{\text{UC}}}^u(f,Z)\leq h_{\text{\text{B}}}^u(f,Z)$, completing the proof of Corollary A.2.
\end{proof}

\subsection{Variational principle for unstable entropies of subsets}
In this subsection, we prove theorem B, i.e., the variational principle for unstable entropies of subsets.
The proof is based on the variational principle for entropies of subsets established by Feng-Huang (cf. \cite{FH}).
Denote
\begin{equation*}
\underline{h}_\mu(f)=\int_M\underline{ h}_\mu(f,x) d\mu
\end{equation*}
where
\begin{equation*}
\begin{aligned}
\underline{h}_\mu(f,x)=\lim_{\e\to 0}\liminf_{n\to \infty}-\frac{1}{n}\log\mu(B_{n}(x,\epsilon)).
\end{aligned}
\end{equation*}

\begin{theorem}\label{subsetvp}(Cf. \cite{FH})
Let $f: M\to M$ be a continuous map on a compact metric space $M$. For any nonempty compact subset $K\subset M$, we have
$$h_{\text{\text{B}}}(f,K)=\sup\{\underline{h}_{\mu}(f): \mu\in \M(M) \  \text{and\ } \mu(K)=1\}.$$
\end{theorem}

Fix an arbitrarily small $\rho>0$.
For some $\delta>0$ small enough, by Lemma \ref{smalldelta} we can find a point $x\in M$ such that
\begin{equation}\label{e:approach1}
\begin{aligned}h_{\text{\text{B}}}(f, K\cap \overline{W^u(x,\delta)})\geq h^u_{\text{\text{B}}}(f,K)-\rho/2.
\end{aligned}
\end{equation}
By Theorem \ref{subsetvp}, there exists $\mu\in \M(M)$ such that $\mu(K\cap \overline{W^u(x,\delta)})=1$
and
\begin{equation}\label{e:approach2}
\begin{aligned}
\underline{h}_\mu(f)\geq h_{\text{\text{B}}}(f, K\cap \overline{W^u(x,\delta)})-\rho/2.
\end{aligned}
\end{equation}

As $\delta$ is very small, we can choose a partition $\eta\in\P^u$ such that $\overline{W^u(x,\delta)}\subset \eta(x)$.
That is, $\overline{W^u(x,\delta)}$ is contained in a single element of $\eta$.
Since $\text{supp}(\mu)\subset \overline{W^u(x,\delta)}\subset \eta(x)$, $\eta$ is a trivial partition with respect to $\mu$.
Therefore, $\mu=\mu_x^\eta$. Comparing definitions of $\underline{h}_\mu(f)$ and $\underline{h}^u_\mu(f)$ and using Lemma \ref{independence}, we have $\underline{h}_\mu(f)=\underline{h}^u_\mu(f)$.
By \eqref{e:approach1} and \eqref{e:approach2},
\begin{equation*}
\begin{aligned}
\underline{h}_\mu^u(f)\geq h^u_{\text{\text{B}}}(f, K)-\rho.
\end{aligned}
\end{equation*}
As $\rho$ is arbitrary, we know that $h_{\text{\text{B}}}(f,K)\leq \sup\{\underline{h}^u_{\mu}(f): \mu\in \M(M) \ \text{and\ } \mu(K)=1\}$.

For the other direction, let $\mu\in \M(M)$ with $\mu(K)=1$.
Then by Corollary A.1, $h^u_{\text{\text{B}}}(f,K)\geq \underline{h}^u_{\mu}(f)$.
The proof of Theorem B is complete.

\section{Unstable entropies of saturated sets for general PHDs}
In this section, we mainly prove Theorem C. We recall some useful notions at first.
Two points $y,z\in \overline{W^u(x,\delta)}$ are called \emph{$(\rho, n,\epsilon)$ $u$-separated} if
\[\#\{j\in \NN: d^u(f^jy,f^jz)>\e, 0\leq j\leq n-1\}\geq \rho n.\]
A set $S\subset \overline{W^u(x,\delta)}$ is \emph{$(\rho, n,\epsilon)$ $u$-separated} if any pair of distinct points of $S$ are $(\rho, n,\epsilon)$ $u$-separated.
Let $F\subset \M(M)$ be a neighborhood of $\nu\in \M_f(M)$. Denote
$$M_{n,F}:=\{x\in M: \CE_n(x)\in F\}$$
where $\CE_n(x):=\frac{1}{n}\sum_{i=0}^{n-1}\d_{f^i(x)}$.

In \cite{PS2007}, the authors define $N(F; \e,n)$ to be the maximal cardinality of an
$(n,\epsilon)$ separated subset of $ M_{n,F}$
and $N(F; \e,n, \rho)$ to be the maximal cardinality of an
$(\rho, n,\epsilon)$ separated subset of $M_{n,F}$.
To study unstable entropy, we define
$N^u(F; \e,n,x,\d)$ to be the maximal cardinality of an
$(n,\epsilon)$ $u$-separated set of $\overline{W^u(x,\delta)}\cap M_{n,F}$
and $N^u(F; \e,n, \rho, x,\d)$ to be the maximal cardinality of an
$(\rho, n,\epsilon)$ $u$-separated set of $\overline{W^u(x,\delta)}\cap M_{n,F}$.
It is easy to see that $N(F; \e,n)\geq N^u(F; \e,n, x,\delta)$ and $N(F; \e,n, \rho)\geq N^u(F; \e,n, \rho,x,\delta)$
for any $x\in M$ when $\d$ is small enough.

\begin{definition}\label{PSentro}
For $\nu\in \M_f(M)$, $\e>0$, define
\begin{equation*}
\begin{aligned}
\underline{s}(\nu; \e, x, \d):&=\inf_{F\ni \nu}\liminf_{n\to \infty}\frac{1}{n}\log N^u(F; \e,n,x,\d),\\
\underline{s}(\nu; x, \d):&=\lim_{\e\to 0}\underline{s}(\nu; \e, x, \d),\\
\underline{s}(\nu):&=\sup_{x\in M}\underline{s}(\nu; x, \d).
\end{aligned}
\end{equation*}
$\overline{s}(\nu; \e, x, \d)$, $\overline{s}(\nu; x, \d)$ and $\overline{s}(\nu)$ are defined similarly.
\end{definition}
It is standard to verify that $\underline{s}(\nu)$ and $\overline{s}(\nu)$ are independent of $\d>0$ when $\d$ is small enough.
\begin{proposition}\label{vpleq}
For any $\mu\in \M_f(M)$, we have
\[\overline{s}(\mu)\leq h^u_\mu(f).\]
\end{proposition}
\begin{proof}
Assume the contrary that $\overline{s}(\mu)> h^u_\mu(f)$.
By definition there exists $x\in M$ such that
\[\lim_{\e\to 0}\inf_{F\ni \nu}\limsup_{n\to \infty}\frac{1}{n}\log N^u(F; \e,n,x,\d)> h^u_\mu(f).\]
Then there exist $\e_0>0$ and $\rho>0$ such that for any $0<\e<\e_0$,
\[\inf_{F\ni \nu}\limsup_{n\to \infty}\frac{1}{n}\log N^u(F; \e,n,x,\d)\geq h^u_\mu(f)+2\rho.\]
Hence there exists a sequence of nested convex closed neighborhoods $C_1\supset C_2 \supset \cdots$ such that
$\bigcap_{n}C_n=\{\mu\}$ and
\begin{equation}\label{e:geq}
\limsup_{n\to \infty}\frac{1}{n}\log N^u(C_n; \e,n,x,\d)\geq h^u_\mu(f)+2\rho.
\end{equation}
Let $S_n$ be an $(n,\e)$ $u$-separated subset of $\overline{W^u(x,\delta)}\cap M_{n,C_n}$ with maximal cardinality and define
\[\mu_n:=\frac{1}{\#S_n}\sum_{z\in S_n}\CE_n(z).\]
Then $\lim_{n\to \infty}\mu_n=\mu$ by the choice of $C_n$. By Lemma \ref{vplemma},
we have
\[\limsup_{n\to \infty}\frac{1}{n}\log \#S_n=\limsup_{n\to \infty}\frac{1}{n}\log N^u(C_n; \e,n,x,\d)\leq h^u_\mu(f),\]
contradicting \eqref{e:geq}.
\end{proof}

\begin{proof}[Proof of Theorem  C]
Denote $s=\sup\{h^u_\mu(f): \mu\in K\}$. Let $s'-s=2\rho>0$. Since $N(F; \e,n,x,\d)$ is a non-increasing function of $\e$,
by Proposition \ref{vpleq} we have
\[\inf_{F\ni \nu}\limsup_{n\to \infty}\frac{1}{n}\log N^u(F; \e,n,x,\d)\leq h^u_\mu(f)\]
for any $x\in M, \e>0$. Fix $x\in M$. There exist a neighborhood of $\mu$, $F(\mu,\e)$ and $N(F(\mu,\e))\in \NN$ such that
\[\frac{1}{n}\log N^u(F(\mu,\e); \e,n,x,\d)\leq h^u_\mu(f)+\rho\]
for all $n\geq N(F(\mu,\e))$. As a maximal $(n,\e)$ $u$-separated subset of a set $A\subset \overline{W^u(x,\delta)}$ is also $(n,\e)$ $u$-spanning subset of $A$, we know for any  $n\geq N(F(\mu,\e))$,
\[\M_{n,\e}^{s'}(\overline{W^u(x,\delta)}\cap M_{n, F(\mu,\e)})\leq  N^u(F(\mu,\e); \e,n,x,\d)e^{-s'n}\leq e^{-\rho n}.\]

Since $K$ is compact, there exists a finite open cover of $K$ by sets $F(\mu_j,\e), j=1,\cdots, m_\e$, with $\mu_j\in K$. Recall that $\prescript{K}{}{G}=\{x:V_f(x)\cap K\neq \emptyset\}$. Then
\[\prescript{K}{}{G} \subset \bigcup_{n\geq N}\bigcup_{j=1}^{m_\e}M_{n, F(\mu_j,\e)}\]
for arbitrarily large enough $N$. Then for $N>\max_{j}N(F(\mu_j,\e))$,
\[\M_{N,\e}^{s'}(\overline{W^u(x,\delta)}\cap\prescript{K}{}{G})\leq m_\e\sum_{n\geq N}e^{-\rho n}.\]
Since $s'>s$ is arbitrary, it follows that $h_B(f,\overline{W^u(x,\delta)}\cap\prescript{K}{}{G})\leq \sup\{h^u_\mu(f): \mu\in K\}$ for any $x\in M$. Thus
$h^u_B(f,\prescript{K}{}{G})\leq \sup\{h^u_\mu(f): \mu\in K\}$.
\end{proof}

\begin{proof}[Proof of Corollary  C.1]
If $\mu\in \M_f(M)$, then
\[h^u_B(f, G_\mu)\leq h^u_B(f, \prescript{\{\mu\}}{}{G})\leq h^u_\mu(f)\]
by Theorem C.
If $\mu$ is ergodic, then $\mu(G_\mu)=1$ by Birkhoff ergodic theorem. So
$h^u_B(f, G_\mu) \geq h^u_\mu(f)$ by Corollary A.1. Thus $h^u_B(f, G_\mu) = h^u_\mu(f)$ whenever $\mu$ is ergodic.
\end{proof}

\begin{proof}[Proof of Corollary  C.2]
Note that $G_K\subset \prescript{\{\mu\}}{}{G}$ for any $\mu\in K$. It follows that
$$h^u_B(f, G_K) \leq \inf\{h^u_\mu(f): \mu\in K\}$$
by Theorem C.
\end{proof}

\begin{proof}[Proof of Corollary  C.3]
Denote
%$t_a=\sup_{\mu\in M(f,  X)}\left\{h^u_\mu(f):\,  \int\varphi d\mu=a\right\}$ and
 $K:=\left\{\mu\in \M_f(M):\,  \int\varphi d\mu=a\right\}$.
 Note that by the definition of weak$^*$ topology, $R_\varphi(a)\subset \prescript{K}{}{G}$. Thus by Theorem C we complete the proof.
\end{proof}

\section{Saturated sets for PHDs with  $g$-almost product  property}
In this section, we prove Theorem D. Then we obtain a corollary on the
topological entropy of completely irregular sets, and a formula for the unstable metric entropy of ergodic measures.

A PHD $f:M\to M$ has the \emph{unstable uniform separation property}
if the following holds: For any $\kappa>0$, there exist $\rho^*>0$ and $\e^*>0$ such that for any ergodic $\mu$
and any neighborhood $F\subset \M(M)$ of $\mu$, there exists $n^*_{F,\mu,\kappa}\in \NN$ such that for $n\geq n^*_{F,\mu,\kappa}$,
$$\text{ess}\sup_{\mu}N^u(F; \e^*,n, \rho^*, x,\d)\geq e^{n(h_\mu^u(f)-\kappa)}.$$

\begin{proof}[Proof of Theorem D]
 (1) By Theorem A in \cite{HHW}, if $\mu$ is an ergodic measure, then for any $\a\in \P$, $\eta\in \P^u$,
$$h_\mu^u(f)=h_\mu(f, \a|\eta).$$

Let $\kappa>0$. Choose $\e>0$ small enough so that any finite Borel partition of $M$ lies in $\P$ whenever its diameter is smaller than $\e$. Fix a partition $\a=\{A_1, \cdots, A_k\}$ of $M$ with $\diam(\a)<\e$. We choose $0<\kappa'<\kappa$ and a constant $\rho^*>0$ such that
$2\kappa'+\rho^*<1/2$ and
\begin{equation}\label{e:kappa}
\begin{aligned}
\phi(\rho^*+2\kappa')+(\rho^*+2\kappa')\log(2k-1)<\kappa-\kappa',
\end{aligned}
\end{equation}
where $\phi(\rho):=-\rho\log \rho-(1-\rho)\log (1-\rho)$.

Let $\nu\in \M_f(M)$. We will first construct a neighborhood of $\nu$, $W_\nu\subset \M(M)$.
It is enough to prove the theorem for any ergodic $\mu\in W_\nu\cap \M_f(M)$, since $\M_f(M)$ can be covered
by a finite number of neighborhoods $W_{\nu_i}\subset \M(M), i=1, \cdots, n$.

Since $\nu$ is regular, there exist compact subsets $V_j\subset A_j, j=1, \cdots, k$ with
$$\nu(A_j\setminus V_j)<\frac{\kappa'}{4k\log(2k)}.$$
Choose $\e^*>0$ such that $d(x,y)>2\e^*$ as long as $x\in V_i, y\in V_j$, and $i\neq j$.
Let $n_1^*\in \NN$ such that $$\frac{\kappa'}{4\log k}\geq \frac{\log 2}{n\log (2k)}$$
for any $n\geq n_1^*$.
Now let $U_i\supset V_i, i=1,\cdots,k$ be $k$ open sets such that $\diam (U_i)<\e$ and $d(x,y)>\e^*$ whenever $x\in U_i, y\in U_j$, and $i\neq j$. Define a partition $\bar{\a}=\{U_1, \cdots, U_k, A_1\setminus \bigcup_jU_j, \cdots, A_k\setminus \bigcup_jU_j\}$. We use $1, \cdots, 2k$ to label the elements of $\bar{\a}$ in the order listed above. Then elements of $\bar{\a}_0^{n-1}$ can be labeled by words $w$ of length $n$ over the alphabet $\texttt{A}=\{1, \cdots, 2k\}$. Define a map $w: M\to \texttt{A}^n$ by $w(x)_j:=i$ if $f^j(x)$ is in the element of $\bar{\a}$ labeled by $i$.

Denote $K:=M\setminus \bigcup_jU_j$ and define
\[W_\nu:=\left\{\lambda\in \M(M): \lambda(K)\leq \nu(K)+\frac{\kappa'}{4\log(2k)}\right\},\]
which is a neighborhood of $\nu$. We prove the theorem for an ergodic $\mu\in W_\nu\cap \M_f(M)$. By the choice of $\e$, $\bar{\a}\in \P$, there exists $n_2^*\in \NN$ such that if $n\geq n_2^*$ such that
\begin{equation}\label{e:entro}
\begin{aligned}
H_\mu(\bar{\a}_0^{n-1}|\eta)>n(h_\mu^u(f)-\kappa'/2).
\end{aligned}
\end{equation}
By the definition of $W_\nu$,
\begin{equation*}
\begin{aligned}
\mu(K)\leq \nu(K)+\frac{\kappa'}{4\log(2k)}\leq \sum_{j=1}^k\nu(A_j\setminus V_j) +\frac{\kappa'}{4\log(2k)}<\frac{\kappa'}{2\log(2k)}.
\end{aligned}
\end{equation*}
Denote
\[Y_n:=\{x\in M: \#\{j\in \NN: w(x)_j>k, 1\leq j\leq n\}\leq n\kappa'\}.\]
Then $\lim_{n\to \infty}\mu(Y_n)= 1$ by Birkhoff ergodic theorem. Recall that $\lim_{n\to \infty}\mu(M_{n,F})= 1$  if $F$ is a neighborhood of $\mu$.
 Denote $B_n:=M_{n,F}\cap Y_n$. Then there exists $n_3^*\in \NN$ large enough such that
\begin{equation}\label{e:entropy}
\begin{aligned}
\mu(M\setminus B_n)<\frac{\kappa'}{2\log(2k)}-\frac{\log 2}{n\log(2k)}
\end{aligned}
\end{equation}
for any $n\geq n_3^*$. Let $n^*=\max\{n_1^*,n_2^*,n_3^*\}$.
Let $\mathcal{A}:=\{B_n, M\setminus B_n\}$ be a partition of $M$.

We claim that there exists $C\subset M$ with $\mu(C)>0$ such that for any $x\in C$, we have
\begin{equation*}
\begin{aligned}
H_{\mu_x^\eta|_{B_n}}(\bar{\a}_0^{n-1})\geq n(h^u_\mu(f)-\kappa').
\end{aligned}
\end{equation*}
Indeed, if we assume the contrary, then for $\mu \ae x\in M$, we have
\begin{equation}\label{e:entropy1}
\begin{aligned}
H_{\mu_x^\eta|_{B_n}}(\bar{\a}_0^{n-1})< n(h^u_\mu(f)-\kappa').
\end{aligned}
\end{equation}
By the definition of unstable metric entropy and conditional measures, we have
\begin{equation*}
\begin{aligned}
H_\mu(\bar{\a}_0^{n-1}|\eta)&=\int_MH_{\mu_x^\eta}(\bar{\a}_0^{n-1})d\mu(x)\\
&\leq \int_MH_{\mu_x^\eta}(\bar{\a}_0^{n-1}|\mathcal{A})+H_{\mu_x^\eta}(\mathcal{A})d\mu(x)\\
&\leq \log 2+ \int_M\mu_x^\eta(M\setminus B_n)H_{\mu_x^\eta|_{M\setminus B_n}}(\bar{\a}_0^{n-1})+\mu_x^\eta(B_n)H_{\mu_x^\eta|_{B_n}}(\bar{\a}_0^{n-1})d\mu(x).
\end{aligned}
\end{equation*}
Note that $H_{\mu_x^\eta|_{M\setminus B_n}}(\bar{\a}_0^{n-1})\leq n\log(2k)$. Using \eqref{e:entro} and \eqref{e:entropy1}, we have
\begin{equation*}
\begin{aligned}
n(h_\mu^u(f)-\kappa'/2)&<H_\mu(\bar{\a}_0^{n-1}|\eta)\\
&\leq \log 2+ n\log(2k)\mu(M\setminus B_n)+n(h^u_\mu(f)-\kappa')\mu(B_n)
\end{aligned}
\end{equation*}
Taking into account \eqref{e:entropy}, we are arriving at a contradiction. This proves the claim.

For $x\in C$, denote $\Xi_n(x)$ be the image of $B_n\cap \eta(x)$ under the map $w$. It follows from the above claim that
\[\#(\Xi_n(x))\geq e^{n(h^u_\mu(f)-\kappa')}.\]
Recall that the Hamming distance $d_n^H(w,w')$ of two words in $\texttt{A}^n$ is the number of different letters of $w$ and $w'$.
Let $\Xi_n'(x)\subset \Xi_n(x)$ be a subset of maximal cardinality such that $d_n^H(w,w')> n(2\kappa'+\rho^*)$
for any distinct $w,w'\in \Xi_n'(x)$. Picking exactly one point from the preimage of each word in $\Xi_n'(x)$ under the map $w$, we get a set $\Gamma_n(x)$. Then it is easy to verify that $\Gamma_n(x)\subset M_{n,F}\cap \eta(x)$ is $(\rho^*, n, \e^*)$-separated.
By the maximal cardinality of $\Xi_n'(x)$, for each $w\in \Xi_n(x)$, there exists $w'\in \Xi'_n(x)$ such that $d_n^H(w,w')\leq n(2\kappa'+\rho^*)$. On the other hand, for a given $w\in \texttt{A}^n$,
\[\#\{w'\in \texttt{A}^n: d_n^H(w,w')\leq n(2\kappa'+\rho^*)\}\leq e^{n\phi(2\kappa'+\rho^*)}(\#\texttt{A}-1)^{n(2\kappa'+\rho^*)}.\]
Combining with \eqref{e:kappa},we have
\[\#\Gamma_n(x)=\#\Xi'_n(x)\geq \frac{e^{n(h^u_\mu(f)-\kappa')}}{e^{n\phi(\rho^*+2\kappa')}(2k-1)^{n(\rho^*+2\kappa')}}\geq e^{n(h^u_\mu(f)-\kappa)}.\]
Since we can choose $\eta\in \P$ and $\d>0$ such that $\eta(x)\subset \overline{W^u(x,\delta)}$ for any $x\in M$, it follows from above that $f$ has the unstable uniform separation property.

(2) We first introduce the following lemma. The definition of $N(F; \e^*,n, \rho^*)$ is given at the beginning of Section 4.
\begin{lemma}\label{lemma-invMeas-uniform}
Let $f: M \to M$ be a $C^1$-PHD and suppose that $f$ has the $g$-almost product  property. Then for any $\kappa>0$, there exist $\rho^*>0$ and $\e^*>0$ such that for any $\mu\in \M_f(M)$
and any neighborhood $F\subset \M(M)$ of $\mu$, there exists $n^*_{F,\mu,\kappa}$ such that for any $n\geq n^*_{F,\mu,\kappa}$,
$$ N(F; \e^*,n, \rho^*)\geq e^{n(h_\mu^u(f)-\kappa)}.$$
\end{lemma}

\begin{proof}  Let $\kappa>0$ and $\mu\in F.$ By item (1)  for $\frac\kappa 3$   there exist $\rho^*>0$ and $\e^*>0$ such that the statement is true for the ergodic measures with $ \frac\kappa 3$ instead of $\kappa$,  $  2{\rho^*}  $ instead of $\rho^*$ and  $ 2{\e^*}  $ instead of $\e^*.$ Note that $N(F; \e,n, \rho)\geq N^u(F; \e,n, \rho,x,\delta)$.  If $\mu$ is not ergodic and $\mu\in F,$ one can use the ergodic decomposition and the $g$-almost product  property to obtain this result by following the construction in \cite{PS2005} to prove entropy-dense property (though we do not know whether unstable version of entropy-dense property is true).

More precisely, by the ergodic decomposition formula \eqref{e:ergodic}, there exist finite $l$ numbers of $\theta_i\in(0,1)$ with $\sum_{i=1}^l\theta_l=1$ and ergodic measures $\mu_i $ such that
$\nu:=\sum_{i=1}^l\theta_i\mu_i\in F$ and $h^u_\nu(f)\geq h^u_\mu(f)-\frac13\kappa.$ For simplicity, we assume $l=2$ and $\theta_1=\theta_2=\frac12$. The proof for the general case is very similar. Take two neighborhoods $F_i$ of $\mu_i$ such that $\frac12\nu_1+\frac12\nu_2\in F$ for any $\nu_i\in F_i$ $(i=1,2)$. Since $\mu_i$ are ergodic, one can choose $n_0$ large enough and $(2\rho^*,n_0,2\e^*)$-separated sets $\Gamma_i\subseteq M_{n_0, F_i}$ such that $ \# \Gamma_i\geq e^{n_0(h_{\mu_i}^u(f)-\frac\kappa 3)}$. Take $\e<\frac14 \e^*$ and we may assume that $n_0$ satisfies
\begin{equation*}\label{Eq-choose-n_k}
n_0>m(\epsilon), \, \rho^* n_0> 2 g(n_0)+1 \text{ and } \frac{g(n_0)}{n_0}\leq \epsilon,
\end{equation*}
where $m(\epsilon)$ and the function $g$ are from the definition of the $g$-almost product property (Definition 2.1 in \cite{PS2007}).
 Now we define the sequences $\{n'_j\}$,  $\{\epsilon'_j\}$ and  $\{\Gamma'_j\}$, by setting
$$n'_j:=n_0,\,\, \epsilon'_j:=\epsilon, \,\,\Gamma'_j:=\Gamma_{j\bmod 2}.$$
Let
$$G_k:=\bigcap_{j=1}^k\left(\bigcup_{x_j\in \Gamma'_j} T^{-M_{j-1}} \overline{B_{n'_j}(g; x_j,\epsilon'_j)}\right)\text{ with } M_j:=\sum_{l=1}^j n'_l.$$
Note that $G_k$ is a non-empty closed set by the $g$-almost product property. We can label each set obtained by developing this formula by the branches
of a labeled tree of height $k$. A branch is labeled by $(x_1,\cdots,x_k)$
with $x_j\in \Gamma'_j.$ Let $G:=\bigcap_{k\geq 1} G_k.$ Then by $2\e^*$-separateness of  $\Gamma'_j$ and $\e<\frac14 \e^*$, $G$ is a closed set that is the disjoint union of non-empty closed sets $G(x_1,x_2,\cdots)$ labeled by
$(x_1,x_2,\cdots)$ with $x_j \in \Gamma'_j$. Two different sequences label two different sets for which any two points in different sets are $(\rho^*,n,\e^*)$-separated provided that $\epsilon$ is small enough and $n>n_0$ large enough. Moreover, one can verify that $x\in M_{n,F}$ for any $x\in G.$ Then, one can check that the cardinality of maximal $(\rho^*,n,\e^*)$-separated sets of $G$  is larger than $$\Pi_{i=1}^{2k}\# \Gamma'_i\geq (e^{n_0(h_{\mu_1}^u(f)-\frac\kappa 3)}e^{n_0(h_{\mu_2}^u(f)-\frac\kappa 3)})^k\geq e^{(n-2n_0)(h^u_\mu(f)-\frac23\kappa)}$$ where $k=[\frac{n}{2n_0}]$. This separated set satisfies the conclusion if $n$ is large enough.
\end{proof}

In view of Lemma \ref{lemma-invMeas-uniform}, one can follow the proof of  \cite[Theorem 1.1]{PS2007} word by word to obtain item (2), by replacing $h_\mu(f)$ by $h^u_\mu(f)$. We omit the details here. This finishes the proof of Theorem D. We emphasize that from u-uniform separation one can only obtain uniform constants  $\rho^*>0$ and $\e^*>0$ for the estimate $ N(F; \e^*,n, \rho^*)\geq e^{n(h_\mu^u(f)-\kappa)}$ but it is unknown for the estimate $ N(F; \e^*,n, \rho^*)\geq e^{n(h_\mu(f)-\kappa)}.$
\end{proof}

The irregular set of ergodic average is also an important topic in the study of multifractal analysis.
For a continuous function $\varphi$  {on $M$},   define the \emph{$\varphi$-irregular set} as
\begin{eqnarray*}
  I_{\varphi}(f):= \left\{ {x\in M:} \lim_{n\to\infty}\frac1n\sum_{i=0}^{n-1}\varphi(f^ix) \ \text{diverges}\right\}.
\end{eqnarray*}
It was proved in \cite{Thompson2012} that if the system has $g$-almost product  property, then for any continuous function $\varphi,$ $I_{\varphi}$ either is empty or has full topological  entropy.
Let $C^*(M,\mathbb{R})$ denote the space of all continuous functions of $\varphi$ with $I_\varphi\neq \emptyset$ and define $CI=\bigcap_{\varphi\in  {C^*(M,\mathbb{R})}}I_\varphi$. The set $CI$ is called the \emph{completely irregular set} of $f$, see \cite{Tian}. It was proved in  \cite{Tian} that if the system has the $g$-almost product  property and uniform separation property, then $CI$ has full topological entropy. Here for PHDs, we give an estimate of topological entropy on $CI$, just provided that the systems has   $g$-almost product  property.

\begin{CorollaryD1}\label{CorD1}
Let $f: M \to M$ be a $C^1$-PHD. If $f$ has the $g$-almost product  property, then
$h_{\text{B}}(f, CI)\geq h^u_\text{B}(f)$.
\end{CorollaryD1}

\begin{proof}
Fix $\epsilon>0$. By the variational principle of unstable entropy in \cite{HHW}, we can take an invariant measure $\mu$ such that $h^u_\mu(f)> h^u_\text{B}(f)-\epsilon.$  Suppose that $I_\varphi\neq \emptyset.$ Then there exist two invariant measures with different integrals of $\varphi$, that is, there exists another invariant measure $\omega_\varphi$ such that $\int\varphi d\mu\neq \int \varphi d\omega_\varphi.$ For $\theta\in (0,1)$, define $\nu_\varphi=\theta\mu+(1-\theta)\omega_\varphi.$ By Proposition 2.14 in \cite{HHW}, $h^u_{\nu_\varphi}(f)=\theta h^u_{\mu}(f)+(1-\theta)h^u_{\nu_\varphi}(f)$. Take $\theta\in (0,1)$ close to 1 such that $h^u_{\nu_\varphi}(f)>  h^u_\text{B}(f)-\epsilon$.  Note that  $\int\varphi d\mu\neq \int \varphi d\nu_\varphi.$  Let $$K:=\overline{\text{convex}(\{\mu\}\cup\{\nu_\varphi|\,\varphi\in {C^*(M,\mathbb{R})} \})}.$$  Note that $K$ is compact and connected, and  $G_K\subset CI.$ By Proposition 2.15 in \cite{HHW}, the unstable entropy map $\nu\mapsto h^u_\nu(f)$ is upper semi-continuous. Then $\inf_{\nu\in K} h^u_\nu(f)\geq  h^u_B(f)-\epsilon.$ Thus by Theorem D,
$$h_{\text{B}}(f,CI)\geq h_{\text{B}}(f,G_K)\geq \inf_{\nu\in K} h^u_\nu(f)\geq  h^u_\text{B}(f)-\epsilon.$$
Letting $\e\to 0$, we complete the proof.
\end{proof}

Recall that $\underline{s}(\mu)$ and $\overline{s}(\mu)$ are defined in Definition \ref{PSentro}.
The following extends the formula given by Pfister-Sullivan (\cite{PS2007}) to the context of unstable metric entropy for ergodic measures.
\begin{CorollaryD2}\label{CorD2}
Let $f: M \to M$ be a $C^1$-PHD and $\mu$ an ergodic measure. Then
$$h_\mu^u(f)=\underline{s}(\mu)=\overline{s}(\mu).$$
\end{CorollaryD2}
\begin{proof}
By Proposition \ref{vpleq}, for any $\mu\in \M_f(M)$, we have $\overline{s}(\mu)\leq h^u_\mu(f).$
By Theorem D, $f$ has the u-uniform separation property, that is, for any $\kappa>0$, there exist $\rho^*>0$ and $\e^*>0$ such that for any ergodic $\mu$ and any neighborhood $F\subset \M(M)$ of $\mu$, there exists $n^*_{F,\mu,\kappa}\in \NN$ such that for $n\geq n^*_{F,\mu,\kappa}$
$$\text{ess}\sup_{\mu}N^u(F; \e^*,n, \rho^*, x,\d)\geq e^{n(h_\mu^u(f)-\kappa)}.$$
It implies that $\underline{s}(\mu)\geq h^u_\mu(f).$ Thus we have $h_\mu^u(f)=\underline{s}(\mu)=\overline{s}(\mu)$ for any ergodic measure $\mu$.
\end{proof}

\ \
\\[-2mm]
\textbf{Acknowledgement.}  X. Tian is supported by NSFC (11671093) and W. Wu is supported by NSFC (11701559 and 11571387).
W. Wu would like to thank Fudan University for hospitality where part of this work was done.

\end{document}